\documentclass[12pt]{amsart}
\usepackage{fullpage,amssymb,mathtools}
\newcommand{\bigboxplus}{
  \mathop{
    \vphantom{\bigoplus} 
    \mathchoice
      {\vcenter{\hbox{\resizebox{\widthof{$\displaystyle\bigoplus$}}{!}{$\boxplus$}}}}
      {\vcenter{\hbox{\resizebox{\widthof{$\bigoplus$}}{!}{$\boxplus$}}}}
      {\vcenter{\hbox{\resizebox{\widthof{$\scriptstyle\oplus$}}{!}{$\boxplus$}}}}
      {\vcenter{\hbox{\resizebox{\widthof{$\scriptscriptstyle\oplus$}}{!}{$\boxplus$}}}}
  }\displaylimits 
}
\DeclareMathOperator{\sym}{sym}

\DeclareMathOperator{\GL}{GL}
\DeclareMathOperator{\SL}{SL}
\DeclareMathOperator{\Res}{Res}
\DeclareMathOperator{\cond}{cond}
\DeclareMathOperator{\Ad}{Ad}
\DeclareMathOperator{\sgn}{sgn}
\newcommand{\Q}{\mathbb{Q}}
\newcommand{\C}{\mathbb{C}}
\newcommand{\F}{\mathbb{F}}
\newcommand{\Z}{\mathbb{Z}}
\newcommand{\A}{\mathbb{A}}
\newcommand{\R}{\mathbb{R}}
\newtheorem{theorem}{Theorem}
\newtheorem{lemma}[theorem]{Lemma}

\theoremstyle{remark}

\numberwithin{theorem}{section}
\numberwithin{equation}{section}
\begin{document}
\author{Andrew R.~Booker}
\address{School of Mathematics, University of Bristol,
University Walk, Bristol, BS8 1TW, United Kingdom}
\email{andrew.booker@bristol.ac.uk}
\thanks{The author was partially supported by EPSRC Grant
\texttt{EP/K034383/1}. No data were created in the course of this study.}
\title{A note on Maass forms of icosahedral type}
\begin{abstract}
Using ideas of Ramakrishnan, we consider the icosahedral analogue of
the theorems of Sarnak and Brumley on Hecke--Maass newforms with Fourier
coefficients in a quadratic order.  Although we are unable to conclude the
existence of an associated Galois representation in this case, we show
that one can deduce some implications of such an association, including
weak automorphy of all symmetric powers and the value distribution of
Fourier coefficients predicted by the Chebotarev density theorem.
\end{abstract}
\maketitle
\section{Introduction}
In \cite{sarnak}, Sarnak showed that a Hecke--Maass newform with integral
Fourier coefficients must be associated to a dihedral or tetrahedral Artin
representation. Brumley \cite{brumley} later generalized this to
Galois-conjugate pairs of
forms with coefficients in the ring of integers of $\Q(\sqrt{d})$
for a fundamental discriminant $d\ne5$, which are associated to
dihedral, tetrahedral or octahedral representations. In this note we
consider the remaining case of nondihedral forms with coefficients in
$\Z\bigl[\frac{1+\sqrt5}2\bigr]$, which are predicted to correspond to
icosahedral representations.

The results of Sarnak and Brumley depend crucially on the existence and
cuspidality criteria of the symmetric cube and symmetric fourth power
lifts from $\GL(2)$, as established by Kim and Shahidi \cite{KS1,kim,KS2}.
For the icosahedral case, in order to conclude the existence of an
associated Artin representation we would need to know the expected
cuspidality criterion for the symmetric sixth power lift (which is not yet
known to be automorphic). Appealing to ideas and results of
Ramakrishnan \cite{R1,R2,R3}, we show that one can nevertheless derive
some of the consequences entailed by the existence of an associated
icosahedral representation, including weak automorphy of all symmetric
powers and the value distribution of Fourier coefficients predicted by
the Chebotarev density theorem. Our precise result is as follows.
\begin{theorem}\label{thm:main}
Let $\A$ be the ad\`ele ring of $\Q$, and
put $A=\{0,\pm1,\pm2,\pm\varphi,\pm\varphi^\tau\}$,
where $\varphi=\frac{1+\sqrt5}2$ and $\tau$ denotes the nontrivial
automorphism of $\Q(\varphi)$.
Let $\pi=\bigotimes\pi_v$ and $\pi'=\bigotimes\pi_v'$ be nonisomorphic
unitary cuspidal automorphic representations of
$\GL_2(\A)$ with normalized Hecke eigenvalues $\lambda_\pi(n)$ and
$\lambda_{\pi'}(n)$, respectively. Assume that $\pi$ and $\pi'$ are not
of dihedral Galois type, and suppose that
$\lambda_\pi(n)$ and $\lambda_{\pi'}(n)$ are elements of $\Z[\varphi]$
satisfying $\lambda_{\pi'}(n)=\lambda_\pi(n)^\tau$ for every $n$.
Then:
\begin{enumerate}
\item $\pi$ corresponds to a Maass form of weight $0$ and trivial
nebentypus character.
\item For any place $v$, $\pi_v$ is tempered if and only
if $\pi_v'$ is tempered.
\item If $S$ denotes the set of primes $p$ at
which $\pi_p$ is not tempered, then
\begin{enumerate}
\item $\#\{p\in S:p\le X\}\ll X^{1-\delta}$ for some $\delta>0$;
\item $\pi_v\cong\pi_v'$ for all $v\in S\cup\{\infty\}$;
\item $\lambda_\pi(p)\in A$ for every prime $p\notin S$;
\item for each $k\ge0$, there is a unique isobaric automorphic representation
$\Pi_k=\bigotimes\Pi_{k,v}$ of $\GL_{k+1}(\A)$ satisfying
$\Pi_{k,p}\cong\sym^k\pi_p$ for all primes $p\notin S$ at which $\pi_p$
is unramified;
\item if $S$ is finite or $\sym^5\pi$ is automorphic then $S=\emptyset$
and $\pi_\infty$ is of Galois type (so that $\pi$ corresponds to a Maass
form of Laplace eigenvalue $\frac14$).
\end{enumerate}
\item
For each $\alpha\in A$, the set of primes $p$ such that
$\lambda_\pi(p)=\alpha$ has a natural density, depending
only on the norm $\alpha\alpha^\tau$, as follows:
\begin{center}
\begin{tabular}{r|rrrr}
$\alpha\alpha^\tau$ & $0$ & $1$ & $4$ & $-1$ \\ \hline
$\mathrm{density}$ & $\frac14$ & $\frac16$ & $\frac1{120}$ & $\frac1{10}$
\end{tabular}
\end{center}
\end{enumerate}
\end{theorem}

\section{Preliminaries}\label{sec:prelim}
Before embarking on the proof of Theorem~\ref{thm:main},
we first recall some facts about the analytic properties of the standard
and Rankin--Selberg $L$-functions associated to isobaric automorphic
representations. We refer to \cite[\S1]{R1} for essential background and
terminology.

Let $\pi=\sigma_1\boxplus\cdots\boxplus\sigma_n=\bigotimes\pi_v$ be
an isobaric automorphic representation of $\GL_d(\A)$ for some $d\ge1$,
and assume that the cuspidal summands $\sigma_i$ have finite-order central
characters. Then for any finite set of places $S\supseteq\{\infty\}$,
the partial $L$-function $L^S(s,\pi)=\prod_{v\notin S}L(s,\pi_v)$
converges absolutely for $\Re(s)>1$ and
continues to an entire function, apart from a possible pole at $s=1$
of order equal to the number of occurrences of the trivial character
among the $\sigma_i$. Furthermore, $L^S(s,\pi)$ has no zeros in the
region $\{s\in\C:\Re(s)\ge1\}$.

Given two such isobaric representations, $\pi_1$ and
$\pi_2$, we can form the irreducible admissible representation
$\pi_1\boxtimes\pi_2=\bigotimes(\pi_{1,v}\boxtimes\pi_{2,v})$, where for
each place $v$, $\pi_{1,v}\boxtimes\pi_{2,v}$ is the functorial tensor
product defined by the local Langlands correspondence.  Then for any set
$S$ as above, $L^S(s,\pi_1\boxtimes\pi_2)
=\prod_{v\notin S}L(s,\pi_{1,v}\boxtimes\pi_{2,v})$
agrees with the partial Rankin--Selberg
$L$-function $L^S(s,\pi_1\times\pi_2)$, which again converges absolutely
for $\Re(s)>1$, continues to an entire function apart from a possible
pole at $s=1$, and does not vanish in $\{s\in\C:\Re(s)\ge1\}$. The order
of the pole is characterized by the facts that (i) it is bilinear with
respect to isobaric sum, and (ii) if $\pi_1$ and $\pi_2$ are cuspidal then
$L^S(s,\pi_1\boxtimes\pi_2)$ has a simple pole if $\pi_1\cong\pi_2^\vee$
and no pole otherwise.

Given an irreducible admissible representation
$\Pi$ of $\GL_d(\A)$, let $\cond(\Pi)$ denote its conductor, and let
$\{c_n(\Pi)\}_{n=1}^\infty$ be the unique sequence of complex numbers
satisfying
$$
-\frac{L'}{L}(s,\Pi)=\sum_{n=1}^\infty\frac{\Lambda(n)c_n(\Pi)}{n^s}
\quad\text{and}\quad c_n(\Pi)=0\text{ whenever }\Lambda(n)=0.
$$
Then $c_n(\Pi)$ is multiplicative in $\Pi$, in the sense that
if $\pi_1$ and $\pi_2$ are isobaric representations as above
then $c_n(\pi_1\boxtimes\pi_2)=c_n(\pi_1)c_n(\pi_2)$ for all $n$ coprime
to $\gcd(\cond(\pi_1),\cond(\pi_2))$.

Similarly, for any isobaric representation $\pi$ of $\GL_d(\A)$ and any 
$k\ge0$, we can form the irreducible admissible
representation $\sym^k\pi=\bigotimes\sym^k\pi_v$.
If $d=2$ and $\pi$ has trivial central character then
for all $n$ coprime to $\cond(\pi)$ we have
$$
c_n(\sym^k\pi)=P_k(c_n(\pi))
$$
for certain polynomials $P_k\in\Z[x]$; in particular,
$$
P_0=1,\;
P_1=x,\;
P_2=x^2-1,\;
P_3=x^3-2x,\;
P_4=x^4-3x^2+1,\;
P_5=x^5-4x^3+3x.
$$

Finally, we recall some standard tools from analytic number theory.
\begin{lemma}\label{lem:positive}
Let $\{c_n\}_{n=1}^\infty$ be a sequence of nonnegative real numbers satisfying
$c_n\ll n^\sigma$ for some $\sigma\ge0$, and put
$$
D(s)=\sum_{n=1}^\infty\frac{c_n}{n^s}
\quad\text{for }\Re(s)>\sigma+1.
$$
Suppose that $(s-1)D(s)$ has analytic continuation to an open set containing
$\{s\in\C:\Re(s)\ge1\}$, and set $r=\Res_{s=1}D(s)$. Then
\begin{enumerate}
\item
$\sum_{n\le X}c_n=rX+o(X)
\quad\text{as }X\to\infty.$
\item
If $r=0$ then there exists $\delta>0$ such that
$\sum_{n=1}^\infty c_n/n^{1-\delta}<\infty.$
\end{enumerate}
\end{lemma}
\begin{proof}
These are the Wiener--Ikehara theorem \cite[Ch.~II.7, Thm.~11]{tenenbaum}
and Landau's theorem \cite[Ch.~II.1, Cor.~6.1]{tenenbaum}, respectively.
\end{proof}

\section{Proof of Theorem~\ref{thm:main}}\label{sec:proof}
We are now ready for the proof. Our argument is sequential, but to aid
the reader we have separated it into six main steps, as follows.

\subsection{Initial observations}
Let $\pi$ and $\pi'$ be as in the statement of the theorem.
If $\pi_p$ and $\pi_p'$ are both tempered for some prime $p$, then
$\max\{|\lambda_\pi(p)|,|\lambda_\pi(p)^\tau|\}\le2$, which holds 
if and only if $\lambda_\pi(p)\in A$. Thus, conclusion (2) of the
theorem implies conclusion (3c).
Moreover, the equality $\lambda_{\pi'}(n)=\lambda_\pi(n)^\tau$
for $n\in\{p,p^2\}$ implies that the $L$-factors
$L(s,\pi_p)$ and $L(s,\pi_p')$ have the same degree, and thus
$\pi_p$ is ramified if and only if $\pi_p'$ is ramified.
We set $N=\gcd(\cond(\pi),\cond(\pi'))$.

Since $\lambda_\pi(n)\in\R$ for every $n$,
$\pi$ must be self dual. Suppose that $\pi$ is the
automorphic induction of a Hecke character of infinite order.
Then the value distribution of $\lambda_\pi(p)$ for primes $p$
is the sum of a point mass of weight $\frac12$ at $0$ and a
continuous distribution.
In particular, the set $\{p:\lambda_\pi(p)\in A\}$
has density $\frac12$. On the other hand,
by \cite[Theorem~4.1]{KS2}, the set of $p$ at which $\pi_p'$
is tempered has lower Dirichlet density at least $34/35$, and as observed above,
$\lambda_\pi(p)\in A$ for any such $p$. This is a contradiction, so
$\pi$ cannot be induced from a Hecke character of infinite order.
(See \cite{sarnak} for an alternative proof in the Maass form case,
based on transcendental number theory.)
By hypothesis, $\pi$ is also not of dihedral Galois type,
and it follows that $\pi$ has trivial central character.

Suppose that $\pi_\infty$ is a discrete series representation of
weight $k\ge2$. Since $\pi$ has trivial central character, $k$ must
be even. Then $\pi$ corresponds to a holomorphic newform with Fourier
coefficients $\lambda_\pi(n)n^{(k-1)/2}$, which must lie in a fixed number
field. Considering primes $n=p$, since $\lambda_\pi(p)\in\Z[\varphi]$
that is only possible if $\lambda_\pi(p)=0$ for all but finitely many
$p$, contradicting the fact that $L(s,\pi\boxtimes\pi)$ has a pole at
$s=1$. Thus, $\pi_\infty$ must be a principal or complementary series
representation of weight $0$, which establishes (1).

Next suppose that $\pi$ is of tetrahedral or octahedral type. Then
$\Ad(\pi)\cong\sym^2\pi$ corresponds to an irreducible $3$-dimensional
Artin representation with Frobenius traces $\lambda_\pi(p)^2-1$ for all
primes $p\nmid N$, and image isomorphic to $A_4$ or $S_4$, respectively.
In the tetrahedral case, from the character table of $A_4$ we see that
$\lambda_\pi(p)^2-1\in\{3,-1,0\}$, so that $\lambda_\pi(p)\in\Z$;
by strong multiplicity one, that contradicts the hypothesis that
$\pi\not\cong\pi'$. In the octahedral case, from the character table of
$S_4$ and the Chebotarev density theorem, $\lambda_\pi(p)^2-1=1$ for
a positive proportion of primes $p$; that contradicts the hypothesis
that $\lambda_\pi(p)\in\Z[\varphi]$.

In summary, we have shown that $\pi$ corresponds to a Maass form of
weight $0$ and trivial character, is not in the image of automorphic
induction and is not of solvable polyhedral type. By symmetry these
conclusions apply to $\pi'$ as well. Moreover, by Atkin--Lehner
theory, if there is a prime $p\mid N$ with $p^2\nmid N$ then
$\lambda_\pi(p)^2=\lambda_{\pi'}(p)^2=1/p$. That contradicts
the hypothesis that $\lambda_\pi(n)$ and $\lambda_{\pi'}(n)$ are
algebraic integers, so for every $p\mid N$ we must have $p^2\mid N$ and
$\lambda_\pi(p)=\lambda_{\pi'}(p)=0$.

\subsection{Equivalence of $\sym^3\pi$ and $\sym^3\pi'$}\label{sec:equiv}
By the seminal works of Gelbart--Jacquet \cite{GJ},
Ramakrishnan \cite{R0}, Kim--Shahidi \cite{KS1} and Kim \cite{kim}, we
know that the representations $\sym^k\pi$ and $\sym^k\pi'$ for $k\le4$,
$\pi\boxtimes\pi'$, $\pi\boxtimes\sym^2\pi'$ and $\pi'\boxtimes\sym^2\pi$ are
all automorphic.
Moreover, since $\pi$ and $\pi'$ are not in the image of automorphic
induction and are not of solvable polyhedral type, $\sym^k\pi$ and
$\sym^k\pi'$ are cuspidal for $k\le4$.

For brevity of notation, we set
$$
a_n=c_n(\pi)\quad\text{and}\quad b_n=a_n^\tau=c_n(\pi').
$$
Note that $a_p=\lambda_\pi(p)$ for all primes $p$, and
$a_n=0$ whenever $(n,N)>1$.
For $f\in\R[x,y]$, let
$$
D_f(s)=\sum_{n=1}^\infty
\frac{\Lambda(n)f(a_n,b_n)}{n^s}
=\sum_{(n,N)=1}\frac{\Lambda(n)f(a_n,b_n)}{n^s}
+f(0,0)\sum_{p\mid N}\frac{\log{p}}{p^s-1}.
$$
For any $f$ such that $(s-1)D_f(s)$ has an analytic continuation to an open
set containing $\{s\in\C:\Re(s)\ge1\}$, we define
$$
r(f)=\Res_{s=1}D_f(s).
$$
In particular, by the properties of Rankin--Selberg $L$-functions
described in \S\ref{sec:prelim}, $r(P_i(x)P_j(y))$ is defined for
$i,j\le 4$. Note also that $r(f)$ is linear in $f$.

Consider
\begin{equation}\label{eq:F}
\begin{aligned}
F&=(x-y)^2((x-y)^2-5)\\
&=P_4(x)-4P_3(x)y+6P_2(x)P_2(y)-4xP_3(y)+P_4(y)+4P_2(x)-6xy+4P_2(y).
\end{aligned}
\end{equation}
Note that for $u,v\in\Z$ with $u\equiv v\pmod2$, we have
$$
F\biggl(\frac{u+v\sqrt5}2,\frac{u-v\sqrt5}2\biggr)
=25v^2(v^2-1)\ge0.
$$
Since $\sym^k\pi$ and $\sym^k\pi'$ are cuspidal for $k\le4$
and $\pi\not\cong\pi'$, we have $r(F)=6r(P_2(x)P_2(y))$.

Suppose that $\sym^2\pi\cong\sym^2\pi'$.
Then $a_n=\pm b_n$ for all $n$; writing $a_n=\frac{u_n+v_n\sqrt5}2$ as
above, it follows that $2\mid v_n$, so that
$$
F(a_n,b_n)=25v_n^2(v_n^2-1)\ge75v_n^2=15(a_n-b_n)^2.
$$
This implies
$$
6=r(F)\ge15r((x-y)^2)=30,
$$
which is absurd. Hence, $\sym^2\pi\not\cong\sym^2\pi'$ and $r(F)=0$.

By \cite[Theorem~B]{wang}, this in turn implies that
$\pi\boxtimes\sym^2\pi'$ and $\pi'\boxtimes\sym^2\pi$ are cuspidal.
Also, in view of the identity
$$
(xy)^2=(P_2(x)+1)(P_2(y)+1)=P_2(x)P_2(y)+P_2(x)+P_2(y)+1,
$$
we have $r((xy)^2)=1$, so that $\pi\boxtimes\pi'$ is cuspidal.

Since $F(a_n,b_n)$ is nonnegative, by Lemma~\ref{lem:positive}(2)
there exists $\varepsilon>0$ such that
$$
\sum_{n=1}^\infty\frac{\Lambda(n)F(a_n,b_n)}{n^{1-\varepsilon}}<\infty.
$$
Applying Cauchy--Schwarz and the inequality
$$
(x-y)^4F(x,y)=(x-y)^6((x-y)^2-5)\le(x-y)^8\le128(x^8+y^8),
$$
we have
\begin{align*}
\biggl(\sum_{n=1}^\infty\frac{\Lambda(n)}{n^{1-\varepsilon/3}}
(a_n-b_n)^2F(a_n,b_n)\biggr)^2
&\le\sum_{n=1}^\infty\frac{\Lambda(n)}{n^{1-\varepsilon}}
F(a_n,b_n)\cdot\sum_{n=1}^\infty\frac{\Lambda(n)}{n^{1+\varepsilon/3}}
(a_n-b_n)^4F(a_n,b_n)\\
&\le128\sum_{n=1}^\infty\frac{\Lambda(n)}{n^{1-\varepsilon}}F(a_n,b_n)
\cdot\sum_{n=1}^\infty\frac{\Lambda(n)}{n^{1+\varepsilon/3}}
(a_n^8+b_n^8).
\end{align*}
Noting that $x^8=(P_4(x)+3P_2(x)+2)^2$, the final sum on the right-hand side
converges, by Rankin--Selberg.
Thus, the series defining $D_{(x-y)^2F}(s)$ converges absolutely
for $\Re(s)\ge1-\varepsilon/3$, so that $r((x-y)^2F)=0$.

Next, we compute that
\begin{align*}
20P_3(x)P_3(y)&=20-(x-y)^2F(x,y)+P_2(x)P_4(x)-6P_3(x)\cdot yP_2(x)+15P_4(x)P_2(y)\\
&\quad+15P_2(x)P_4(y)-6xP_2(y)\cdot P_3(y)+P_2(y)P_4(y)+14P_4(x)-38P_3(x)y\\
&\quad+60P_2(x)P_2(y)-38xP_3(y)+14P_4(y)+38P_2(x)-48xy+38P_2(y).
\end{align*}
Evaluating $r$ of both sides and using that $r((x-y)^2F)=0$, we see that
$r(P_3(x)P_3(y))=1$, whence $\sym^3\pi\cong\sym^3\pi'$.
Similarly, we have
$$
xP_2(y)\cdot yP_2(x)=P_3(x)P_3(y)+xP_3(y)+yP_3(x)+xy,
$$
from which it follows that $\pi\boxtimes\sym^2\pi'\cong\pi'\boxtimes\sym^2\pi$.
Also, from
$$
P_4(x)-P_4(y)=(x+y)(P_3(x)-P_3(y)+xP_2(y)-yP_2(x)),
$$
we get $P_4(a_n)=P_4(b_n)$, so that $\sym^4\pi_p\cong\sym^4\pi_p'$
for all $p\nmid N$. By strong multiplicity one,
$\sym^4\pi\cong\sym^4\pi'$.

\subsection{Nontempered and archimedean places}
In view of the identity
$$
x^2(P_3(x)-P_3(y))+(x^2+xy-1)(xP_2(y)-yP_2(x))
=(x-y)(x^2-x-1)(x^2+x-1),
$$
for every $n$ we have either $a_n=b_n\in\Z$ or
$a_n\in\{\pm\varphi,\pm\varphi^\tau\}$.
For any prime $p\nmid N$, it follows that if either of $\pi_p$,
$\pi_p'$ is nontempered then $\pi_p\cong\pi'_p$.

Next we show that this conclusion holds for ramified and
archimedean places as well. If $\pi_p$ is nontempered then,
as explained in \cite[Remark~1]{MR}, $\pi_p$ is a
twist of an unramified complementary series representation,
i.e.\ $\pi_p\cong(|\cdot|_p^s\boxplus|\cdot|_p^{-s})\otimes\chi$
for some $s>0$ and unitary character $\chi$ of $\Q_p^\times$.
Since $\sym^3\pi_p\cong\sym^3\pi_p'$, $\pi_p'$ must also be nontempered,
so we similarly have
$\pi_p'\cong(|\cdot|_p^{s'}\boxplus|\cdot|_p^{-s'})\otimes\chi'$
for some $s'>0$ and unitary character $\chi'$. Thus,
$$
\sym^3(|\cdot|_p^s\boxplus|\cdot|_p^{-s})\otimes\chi^3
\cong\sym^3(|\cdot|_p^{s'}\boxplus|\cdot|_p^{-s'})\otimes(\chi')^3,
$$
from which it follows that $s=s'$ and $\chi^3=(\chi')^3$. Comparing
central characters, we deduce that $\chi=\chi'$, whence $\pi_p\cong\pi_p'$.
Running through this argument again with the roles of $\pi$ and $\pi'$
reversed, we obtain conclusions (2) and (3b) of the theorem for finite
places.

Similarly, we have
$\pi_\infty=(|\cdot|_\R^s\boxplus|\cdot|_\R^{-s})\otimes\sgn^\epsilon$
for some $\epsilon\in\{0,1\}$ and $s\in i\R\cup(0,\frac12)$, and comparing
the parameters of $\sym^3\pi_\infty$ and $\sym^3\pi_\infty'$, we conclude
that $\pi_\infty\cong\pi_\infty'$.

\subsection{Value distribution of $\lambda_\pi(p)$}
Making use of the isomorphism $\pi\boxtimes\sym^2\pi'\cong\pi'\boxtimes\sym^2\pi$,
if $0<j<i\le 8-j$ then
\begin{align*}
r(x^iy^j)&=r(x^{i-2}y^{j-1}(yP_2(x)+y))
=r(x^{i-2}y^{j-1}(xP_2(y)+y))\\
&=r(x^{i-1}y^{j+1})+r(x^{i-2}y^j)-r(x^{i-1}y^{j-1}),
\end{align*}
and similarly with the roles of $x$ and $y$ reversed. By systematic application of this
rule and linearity, we reduce the computation of $r(x^iy^j)$ for $i+j\le8$
to that of $r(P_i(x)P_j(y))$, $r(P_i(x)P_j(x))$ and $r(P_i(y)P_j(y))$ for $i,j\le4$,
all of which are determined from the conclusions obtained in \S\ref{sec:equiv}.
After some computation
we arrive at the following table of values for $r(x^iy^j)$:
\begin{center}
\begin{tabular}{r|rrrrrrrrl}
$i\setminus j$&$0$&$1$&$2$&$3$&$4$&$5$&$6$&$7$&$8$\\ \hline
$0$ & $1$&$0$&$1$&$0$&$2$&$0$&$5$&$0$&$14$\\
$1$ & $0$&$0$&$0$&$0$&$0$&$0$&$0$&$1$\\
$2$ & $1$&$0$&$1$&$0$&$2$&$0$&$6$\\
$3$ & $0$&$0$&$0$&$1$&$0$&$4$\\
$4$ & $2$&$0$&$2$&$0$&$5$\\
$5$ & $0$&$0$&$0$&$4$\\
$6$ & $5$&$0$&$6$\\
$7$ & $0$&$1$\\
$8$ & $14$
\end{tabular}
\end{center}
This enables us to compute $r(f)$ for any $f$ of total degree at most $8$
without having to work out the full expansion as in \eqref{eq:F}.

Consider
$$
H=(xy+1)x^2(x^2-1)(x^2-4).
$$
Then $H(a_n,b_n)\ge0$ for all $n$, and for $p\nmid N$, $\pi_p$ and
$\pi_p'$ are tempered if and only if $H(a_p,b_p)=0$. We verify by the above that
$r(H)=0$, so by Lemma~\ref{lem:positive}(2) there exists $\delta\in(0,1]$
such that $\sum_{n=1}^\infty\Lambda(n)H(a_n,b_n)/n^{1-\delta}<\infty$. Thus,
with $S$ as in the statement of the theorem, we have
\begin{equation}\label{eq:Hestimate}
\sum_{\substack{p\in S\\p\le X}}(\log{p})a_p^8
\le\sum_{\substack{n\le X\\a_n\notin A}}\Lambda(n)a_n^8
\ll\sum_{n\le X}\Lambda(n)H(a_n,b_n)
\le X^{1-\delta}\sum_{n=1}^\infty\frac{\Lambda(n)H(a_n,b_n)}{n^{1-\delta}}
\ll X^{1-\delta}.
\end{equation}
Including the possible contribution from ramified primes, which are
finite in number, we see that $\pi_p$ and $\pi_p'$ are tempered for all
but $O(X^{1-\delta})$ primes $p\le X$, which proves (3a).

Next, for each $\alpha\in A$ we define a polynomial $f_\alpha\in\R[x,y]$
of total degree at most $6$, as follows:
\begin{center}
\begin{tabular}{rll}
$\alpha$ & $f_\alpha$ & $r(f_\alpha)/f_\alpha(\alpha,\alpha^\tau)$\\ \hline
$0$ & $(xy+1)(x^2-1)(x^2-4)$ & $\frac14$ \\
$\pm1$ & $(x^2+y^2-3)x(x+\alpha)(x^2-4)$ & $\frac16$ \\
$\pm2$ & $(xy+1)x(x^2-1)(x+\alpha)$ & $\frac1{120}$ \\
$\pm\varphi,\pm\varphi^\tau$ & $(x-y)(1\pm x\pm y)(x-\alpha^\tau)$ & $\frac1{10}$ \\
\end{tabular}
\end{center}
In each case, we have $f_\alpha(\beta,\beta^\tau)\ge0$ for all
$\beta\in\Z\cup A$, with
$f_\alpha(\beta,\beta^\tau)=0$ for $\beta\in A\setminus\{\alpha\}$
and $f_\alpha(\alpha,\alpha^\tau)>0$.
Also, since $\deg f_\alpha\le 6$, we have
$f_\alpha(a_n,b_n)\ll a_n^6\le a_n^8$ whenever $a_n\notin A$.
Thus, by \eqref{eq:Hestimate} and Lemma~\ref{lem:positive}(1),
\begin{align*}
\sum_{\substack{p\le X\\\lambda_\pi(p)=\alpha}}\log{p}
&=O(X^{1/2})+\sum_{\substack{n\le X\\a_n=\alpha}}\Lambda(n)
=O(X^{1/2}+X^{1-\delta})+\frac1{f_\alpha(\alpha,\alpha^\tau)}
\sum_{n\le X}\Lambda(n)f_\alpha(a_n,b_n)\\
&=\frac{r(f_\alpha)}{f_\alpha(\alpha,\alpha^\tau)}X+o(X)
\quad\text{as }X\to\infty.
\end{align*}
By partial summation, it follows that $\{p:\lambda_\pi(p)=\alpha\}$
has natural density $r(f_\alpha)/f_\alpha(\alpha,\alpha^\tau)$, whose values
are shown in the table. This proves (4).

\subsection{Weak automorphy of symmetric powers}
Let $G=\SL_2(\F_5)$, which is the smallest group supporting a
$2$-dimensional icosahedral representation \cite[\S2]{wang}. Then $G$
has nine irreducible representations, with dimensions $1$, $2$, $2$,
$3$, $3$, $4$, $4$, $5$ and $6$. Their characters all take values in
$\Z[\varphi]$ and can be written as
\begin{align*}
\chi_0&=1,\quad\chi_1=\chi,\quad\chi_2=\chi^\tau,\quad\chi_3=P_2(\chi),\quad\chi_4=P_2(\chi^\tau),\\
\chi_5&=\chi\chi^\tau,\quad\chi_6=P_3(\chi),\quad\chi_7=P_4(\chi),\quad\chi_8=\chi P_2(\chi^\tau),
\end{align*}
where $\chi$ is one of the characters of dimension $2$ and $\chi^\tau$
is its Galois conjugate. (Our numbering scheme is more or less arbitrary,
and was made for notational convenience below.)  The character table is
as follows:
\begin{center}
\begin{tabular}{r|rrrrrrrrr}
&$\begin{psmallmatrix}1&0\\0&1\end{psmallmatrix}$
&$\begin{psmallmatrix}4&0\\0&4\end{psmallmatrix}$
&$\begin{psmallmatrix}3&2\\4&3\end{psmallmatrix}$
&$\begin{psmallmatrix}2&2\\4&2\end{psmallmatrix}$
&$\begin{psmallmatrix}2&0\\0&3\end{psmallmatrix}$
&$\begin{psmallmatrix}4&1\\0&4\end{psmallmatrix}$
&$\begin{psmallmatrix}4&2\\0&4\end{psmallmatrix}$
&$\begin{psmallmatrix}1&1\\0&1\end{psmallmatrix}$
&$\begin{psmallmatrix}1&2\\0&1\end{psmallmatrix}$\\ \hline
$\chi_0$ & $1$&$1$&$1$&$1$&$1$&$1$&$1$&$1$&$1$\\
$\chi_1$  & $2$&$-2$&$1$&$-1$&$0$&$\varphi$&$\varphi^\tau$&$-\varphi$&$-\varphi^\tau$\\
$\chi_2$ & $2$&$-2$&$1$&$-1$&$0$&$\varphi^\tau$&$\varphi$&$-\varphi^\tau$&$-\varphi$\\
$\chi_3$  & $3$&$3$&$0$&$0$&$-1$&$\varphi$&$\varphi^\tau$&$\varphi$&$\varphi^\tau$\\
$\chi_4$ & $3$&$3$&$0$&$0$&$-1$&$\varphi^\tau$&$\varphi$&$\varphi^\tau$&$\varphi$\\
$\chi_5$ & $4$&$4$&$1$&$1$&$0$&$-1$&$-1$&$-1$&$-1$\\
$\chi_6$ & $4$&$-4$&$-1$&$1$&$0$&$1$&$1$&$-1$&$-1$\\
$\chi_7$ & $5$&$5$&$-1$&$-1$&$1$&$0$&$0$&$0$&$0$\\
$\chi_8$ & $6$&$-6$&$0$&$0$&$0$&$-1$&$-1$&$1$&$1$
\end{tabular}
\end{center}

\medskip
Let $\langle\;,\;\rangle$ denote the inner product on $L^2(G)$.
For each $k\ge0$ and $i\in\{0,\ldots,8\}$, let
$$
m_{k,i}=\langle P_k(\chi),\chi_i\rangle
$$
be the multiplicity of $\chi_i$ in the $k$th symmetric power of $\chi$, so that
$$
P_k(\chi)=\sum_{i=0}^8 m_{k,i}\chi_i.
$$
In view of the character table, for any $\alpha\in A$ we may choose $g\in
G$ with $\chi(g)=\alpha$ and evaluate both sides of the above at $g$
to get
$$
P_k(\alpha)=\sum_{i=0}^8 m_{k,i}h_i(\alpha,\alpha^\tau),
$$
where we write
\begin{align*}
h_0&=1,\quad h_1=x,\quad h_2=y,\quad h_3=P_2(x),\quad h_4=P_2(y),\\
h_5&=xy,\quad h_6=P_3(x),\quad h_7=P_4(x),\quad h_8=xP_2(y).
\end{align*}

On the automorphic side, we make the parallel definitions
\begin{align*}
\sigma_0&=1,\quad\sigma_1=\pi,\quad\sigma_2=\pi',\quad\sigma_3=\sym^2\pi,\quad\sigma_4=\sym^2\pi',\\
\sigma_5&=\pi\boxtimes\pi',\quad\sigma_6=\sym^3\pi,\quad\sigma_7=\sym^4\pi,\quad
\sigma_8=\pi\boxtimes\sym^2\pi',
\end{align*}
and we set
$$
\Pi_k=\bigboxplus_{i=0}^8(
\underbrace{\sigma_i\boxplus\cdots\boxplus\sigma_i}_{m_{k,i}\text{ times}})
\quad\text{for }k\ge0.
$$
By construction, if $p\notin S$ and $\pi_p$ is
unramified, then for every power $n=p^j$ we have
$$
c_n(\sym^k\pi)=P_k(a_n)=\sum_{i=0}^8 m_{k,i}h_i(a_n,b_n)
=\sum_{i=0}^8 m_{k,i}c_n(\sigma_i)=c_n(\Pi_k).
$$
Thus, $\sym^k\pi_p\cong\Pi_{k,p}$, as claimed.

It remains to prove the uniqueness of $\Pi_k$. Suppose that
$\Pi_k'$ is another such isobaric representation.
Then for any $i\in\{0,\ldots,8\}$,
$$
\sum_{(n,N)=1}\frac{\Lambda(n)c_n(\Pi_k'\boxtimes\sigma_i)}{n^s}
=\sum_{(n,N)=1}\frac{\Lambda(n)c_n(\Pi_k\boxtimes\sigma_i)}{n^s}
+\sum_{(n,N)=1}\frac{\Lambda(n)(c_n(\Pi_k')-c_n(\Pi_k))c_n(\sigma_i)}{n^s}.
$$
Since $c_n(\Pi_k')=c_n(\sym^k\pi)=c_n(\Pi_k)$ whenever $(n,N)=1$ and $a_n\in A$,
by Cauchy--Schwarz we have
\begin{align*}
\biggl(\sum_{(n,N)=1}&\frac{\Lambda(n)|(c_n(\Pi_k')-c_n(\Pi_k))c_n(\sigma_i)|}{n^{1-\delta/3}}\biggr)^2\\
&\qquad\le\sum_{(n,N)=1}\frac{\Lambda(n)|c_n(\Pi_k')-c_n(\Pi_k)|^2}{n^{1+\delta/3}}
\cdot\sum_{\substack{(n,N)=1\\a_n\notin A}}\frac{\Lambda(n)c_n(\sigma_i)^2}{n^{1-\delta}}.
\end{align*}
Since $\Pi_{k,p}'\cong\sym^k\pi_p$ is tempered for all unramified
$p\notin S$, the cuspidal summands of $\Pi_k'$ must be unitary,
so the first sum on the right-hand side converges by Rankin--Selberg.
As for the second, for $a_n\notin A$ we have $c_n(\sigma_i)^2\ll H(a_n,b_n)$, so it
converges as well. Therefore,
$$
\Res_{s=1}\sum_{(n,N)=1}\frac{\Lambda(n)c_n(\Pi_k'\boxtimes\sigma_i)}{n^s}
=\Res_{s=1}\sum_{(n,N)=1}\frac{\Lambda(n)c_n(\Pi_k\boxtimes\sigma_i)}{n^s}
=m_{k,i},
$$
so $\sigma_i$ occurs as a summand of $\Pi_k'$ with multiplicity $m_{k,i}$.
Since $\Pi_k$ and $\Pi_k'$ are both representations of $\GL_{k+1}(\A)$,
we have $\Pi_k'\cong\Pi_k$, as desired. This establishes (3d).

\subsection{Temperedness and Galois type at $\infty$}
Suppose that $\sym^5\pi$ is automorphic. Then it agrees with an isobaric
representation at all unramified finite places. By the uniqueness of
$\Pi_5$, we must have
$\sym^5\pi_p\cong\Pi_{5,p}=\pi_p\boxtimes\sym^2\pi_p'$ for all
$p\nmid N$. When $\lambda_\pi(p)=\lambda_{\pi'}(p)$,
this implies the relation
$$
\lambda_\pi(p)^5-4\lambda_\pi(p)^3+3\lambda_\pi(p)
=\lambda_\pi(p)(\lambda_\pi(p)^2-1),
$$
so that $\lambda_\pi(p)\in\{0,\pm1,\pm2\}$. Thus, $\pi_p$ is tempered
for all $p\nmid N$. In particular, $S$ is finite.

Finally, suppose that $S$ is finite. Then, by what we have already
shown, $\pi$ is \emph{s-icosahedral} in the sense of Ramakrishnan
\cite{R2}. Appealing to \cite[Theorem~A]{R2}, we conclude that $\pi$
is tempered and $\pi_\infty$ is of Galois type.
This establishes (3e) and concludes the proof.

\bibliographystyle{amsplain}
\bibliography{icos}
\end{document}